\documentclass[12pt,reqno]{amsart}
\usepackage{amssymb,a4wide, xy}
\usepackage{amsmath,amscd,amsthm}
\usepackage{amsfonts}
\xyoption{all}

\newtheorem{theorem}{Theorem}[section]

\newtheorem{lemma}[theorem]{Lemma}
\newtheorem{corollary}[theorem]{Corollary}
\newtheorem{prop}[theorem]{Proposition}

\newtheorem*{con7*}{Conjecture 7*}
\newtheorem*{Remark}{Remark}

\theoremstyle{definition}
\newtheorem{definition}[theorem]{Definition}
\newtheorem{example}[theorem]{Example}

\newtheorem{Fact}[theorem]{Fact}
\newtheorem*{th48}{Theorem 4.8}
\newtheorem*{cor49}{Corollary 4.9}
\newtheorem*{th63}{Theorem 6.3}
\newtheorem*{th64}{Theorem 6.4}
\newtheorem*{cor65}{Corollary 6.5}

\theoremstyle{remark}
\newtheorem{remark}[theorem]{Remark}

\numberwithin{equation}{theorem}

\DeclareMathOperator{\Aut}{Aut}
\DeclareMathOperator{\Inn}{Inn}
\DeclareMathOperator{\coker}{coker}

\title[The Box-Tensor Product]{Two generalizations of the
nonabelian tensor product}

\author{Manuel Ladra \and Viji Z. Thomas}
 \address{M. Ladra\\ Department of Algebra, University of Santiago de Compostela,
15782, Spain.}
 \email {manuel.ladra@usc.es}
\address{V.Z. Thomas \\ School of Mathematics, Tata Institute of Fundamental Research,
Mumbai, Maharashtra 400005, India.}
\email{vthomas@math.tifr.res.in}
\thanks{The first author was supported by Ministerio
de Ciencia e Innovaci\'on (European FEDER support included), grant
MTM2009-14464-C02-01, and by Xunta de Galicia, grant Incite09 207
215 PR}

\begin{document}

\begin{abstract}
The purpose of this paper is two fold. First we introduce the box-tensor product of
two groups as a generalization of the
nonabelian tensor product of groups. We extend various results for nonabelian tensor
products to the box-tensor product such as the finiteness of the product when each
factor is finite. This would give yet another proof of Ellis's theorem on the
finiteness of the nonabelian tensor product of groups when each factor is finite.
Secondly, using the methods developed in proving the finiteness of the box-tensor
product, we prove the finiteness of Inassaridze's tensor product under some
additional hypothesis which generalizes his results on the finiteness of his
product. In addition, we prove an Ellis like finiteness theorem under weaker assumptions,
which is a generalization of his theorem on the finiteness of nonabelian tensor product.
As a consequence, we prove the finiteness of low-dimensional nonabelian
homology groups.
\end{abstract}

\maketitle

\section{Introduction}
R. Brown and J.-L. Loday introduced the nonabelian tensor product $G\otimes H$ for a
pair of groups $G$ and $H$ in \cite{BL1} and \cite{BL2} in the context of an
application in homotopy theory. In \cite{NI}, N. Inassaridze extends the
construction of the nonabelian tensor product as given in \cite{BL2}. There he does
not require that the mutual actions satisfy the compatibility conditions. He
constructs the homology groups of groups with coefficients in any group as the left
derived functors of the nonabelian tensor product defined by him, thereby
generalizing the classical theory of homology of groups. He also gives an application
to algebraic $K$-theory of noncommutative local rings.
It is not known if the tensor product $G\otimes H$ introduced in \cite{NI} is finite
when the two groups $G$ and $H$ are finite.
The topic of this paper is to give a generalization of the nonabelian tensor product
of groups as defined in \cite{BL1} and \cite{BL2}, which we call the box-tensor
product and denote it by $G\boxtimes H$.
We prove the finiteness of $G\boxtimes H$ when both $G$ and $H$ are finite.
As an immediate consequence of this result,  we get a different proof of the
finiteness of $G\otimes H$, which are the topics of the papers \cite{E} and
\cite{VT}.
Using the methods developed in proving the finiteness of the box-tensor product when
each of the factors is finite,
we prove the finiteness of the tensor product (cf. Theorem \ref{niko1}) introduced
by N. Inassaridze in \cite{NI} generalizing Theorems 9 and 12 of \cite{NI1}. Finally, we prove that if $G$ and $H$ are finite groups acting on each other and if the mutual actions are half compatible (cf. Definition \ref{hc}), then $G\otimes H$ is finite (cf. Theorem \ref{niko2}).

Consider two groups $G$ and $H$ acting on themselves and on each other.
In the context of this paper all actions will be by automorphisms.
Let $\Aut(G)$ and $\Aut(H)$ be the automorphism groups of $G$ and $H$, respectively.
Consider $\rho_{G} \colon G\rightarrow \Aut(G)$, $\rho_{H} \colon H \rightarrow
\Aut(H)$, $\sigma_{G}\colon G\rightarrow \Aut(H)$ and $\sigma_{H}\colon H\rightarrow
\Aut(G)$. We will write $^gh$ for $\sigma_{G}(g)(h)$ and $^gg'$ for
$\rho_{G}(g)(g')$ and likewise $^hg$ for $\sigma_{H}(h)(g)$ and $^hh'$
for $\rho_{H}(h)(h')$, where $g,g'\in G$ and $h,h'\in H$. Note that $^gg$ need not
be equal to $g$. As in the case of the nonabelian tensor product, we impose a
compatibility condition on the actions.

\begin{definition}\label{D:1.1}
Let $G$ and $H$ be groups which act on themselves and on each other. The mutual
actions are said to be \emph{fully compatible} if
\begin{equation}\label{fc}
 ^{(^ab)}c=\;^a\big(^b(^{a^{-1}}c)\big)
\end{equation}
for all $a,b,c\in G\cup H$.
\end{definition}

For groups with fully compatible actions, the box-tensor product is then defined as
follows.

\begin{definition}\label{D:1.2}
If $G$ and $H$ are groups which act fully compatibly on each other, then the
\emph{box-tensor
product} $G\boxtimes H$ is the group generated by the symbols $g\boxtimes h$ for
$g\in G$ and $h\in H$ with relations
\begin{align}
gg'\boxtimes h \ = \ & (^gg'\boxtimes \;^gh)(g\boxtimes h), \label{E:1.2.1} \\
g\boxtimes hh' \ = \ & (g\boxtimes h)(^hg\boxtimes \;^hh'), \label{E:1.2.2}
\end{align}
 for all $g,g'\in G$ and $h,h'\in H$.
\end{definition}

We now define actions that are compatible.

\begin{definition}\label{D:1.3}
Let $G$ and $H$ be groups that act on themselves by conjugation and each of which
acts on the other. The mutual actions are said to be \emph{compatible} if
\begin{align}
 ^{(^h g)}h' \ = \ & ^h\big(^g(^{h^{-1}}h')\big) \qquad \mbox{and}  \label{E:1.3.1} \\
 ^{(^g h)}g' \ =  \ & ^g\big(^h(^{g^{-1}}g')\big), \label{E:1.3.2}
\end{align}
for all $g,g'\in G$ and $h,h'\in H$.
\end{definition}

\begin{Fact} If $G$ and $H$ are groups acting on themselves by conjugation and each
of which acts on the other, then the following relations always hold.
\begin{equation}\label{fact}
^{(^gh)}h'=\;^g \big(^h(^{g^{-1}}h')\big) \quad \mbox{and} \quad
^{(^hg)}g'=\;^h\big(^g(^{h^{-1}}g')\big)\,,
\end{equation}
\end{Fact}
\noindent for all $g,g'\in G$ and $h,h'\in H$.

Now we define half compatibility, a terminology that is borrowed from \cite{NI1}.

\begin{definition}\label{hc}
Let $G$ and $H$ be groups acting on each other and acting on themselves by
conjugation. We say that the mutual actions are \emph{half compatible} if either
\eqref{E:1.3.1} or \eqref{E:1.3.2} holds.
\end{definition}

\begin{remark}\label{rmk}
It should be noted that half compatibility defined in \cite{NI1} is different than
the definition given here. We borrow that terminology and use it here merely as a
suggestive name to indicate that the mutual actions are half compatible if it
satisfies one of the two compatibility conditions i.e. satisfies half of the
compatibility conditions.
\end{remark}

With this set up, we now state our main results.

\begin{th48} Let $G$ and $H$ be finite groups acting on each other. If the mutual
actions are fully compatible, then $G\boxtimes H$ is finite.
\end{th48}

As a corollary to the above theorem, we get the finiteness of the nonabelian tensor product when each of the factors is finite, which was first proved by Ellis in \cite{E}. A homology free proof of this finiteness result of Ellis can be found in \cite{VT}

\begin{cor49} Let $G$ and $H$ be finite groups acting on each other. If the mutual actions are compatible, then $G\otimes H$ is finite.
\end{cor49}

In the next two Theorems, $G\otimes H$ indicates the tensor product introduced by N.
Inassaridze in \cite{NI}. The next Theorem generalizes Theorems 9 and 12 of
\cite{NI1}.
\begin{th63}
Let $G$ and $H$ be finite groups acting on themselves by conjugation and each of
which acts on the other. If $G$ acts on $H$ trivially, then $G\otimes H$ is finite.
\end{th63}

We also prove that $G\otimes H$ is finite if the mutual actions are
half compatible. Note that the tensor product thus formed is a special case of Inassaridze's tensor product but it is more general than the nonabelian tensor product introduced in \cite{BL1} and \cite{BL2}. Hence the next theorem is a generalization of Ellis-Thomas's theorem on the finiteness of the nonabelian tensor product as given in \cite{E} and \cite{VT}.

\begin{th64}
Let $G$ and $H$ be finite groups acting on each other. If the mutual actions are half compatible, then $G\otimes
H$ is finite.
\end{th64}

As a result of the above theorem, we obtain finiteness of the low dimensional
nonabelian homology groups.

\begin{cor65}
Let $G$ and $A$ be finite groups acting on each other. If the mutual actions are
half compatible, then the nonabelian homology groups $H_i(G,A)$ are finite when
$i=0,1$.
\end{cor65}

The nonabelian tensor product $G\otimes H$ as introduced in \cite{BL1} and
\cite{BL2} becomes a special case of the box-tensor product
 when the groups act on themselves by conjugation, that is $^gg'=gg'g^{-1}$ for all
$g,g'\in G$
 and $^hh'=hh'h^{-1}$ for all $h,h'\in H$. The special case where $G=H$,
  and $\rho_{G}(G)=\rho_{H}(H)=\sigma_{G}(G)=\sigma_{H}(H)$ is called the box-square.
  It is uniquely defined. Note that Definition \ref{D:1.1} is more restrictive than
Definition \ref{D:1.3}.

In Section \ref{S:2}, we study the different compatibility conditions. In
particular, we show that if $G$ and $H$ are groups acting on themselves by
conjugation and each of which acts on the other, then a compatible action is fully
compatible. We give an example showing that a compatible action need not be fully
compatible. We also give an example of an action which is fully compatible.

The topic of this paper is to extend various results which hold for the nonabelian
tensor product to the box-tensor product.
In Section \ref{S:3} this is done for various general results and a crossed module
associated with the box-tensor product is introduced.
The familiar crossed module associated with the nonabelian tensor product is a
special case of the one introduced here for the box-tensor product.  We also show
that if $\overline{D_H(G)}$ (see Definition \ref{D:deri_devi}) is locally cyclic,
then $G\boxtimes H$ is abelian (Theorem \ref{T:2.11}).
We don't know if $G\boxtimes H$ is abelian when both $G$ and $H$ are abelian.
A natural question arises whether the box-tensor product $G\boxtimes H$ is solvable,
nilpotent when $G$ and $H$ are solvable, nilpotent respectively.
In \cite{KT} the authors provide similar results to the one found in \cite{MV} on
the nilpotency and solvability of the box-tensor product $G\boxtimes H$.

As in the case of the nonabelian tensor product, the question arises if the
box-tensor product of two finite groups is finite. Already in \cite{BL2} it is
established that the nonabelian tensor square of a finite group is finite.
In \cite{E}, Ellis shows that the nonabelian tensor product of two finite groups is
finite, using in his proof part of an exact sequence from homology given in
\cite{BL1} and the fact that the homology of a finite group is finite.
Already in \cite{BJR} and \cite{LCK}, the question is asked if a purely group
theoretic result can be given.
In \cite{VT}, the author provides such a proof for the nonabelian tensor product.
In Section \ref{S:4}, a purely group theoretic proof will be given that the
box-tensor product of two finite groups is finite (Theorem \ref{T:3.8}). This proof
is different from the one given in \cite{VT}.

The topic of Section \ref{S:5} is a general construction for the box-tensor product.
For the nonabelian tensor product such constructions were given by Ellis and Leonard
in \cite{E1} and Rocco in \cite{NR}.
In \cite{E1}, the authors note that their result was not new but an adaptation of
those given in \cite{GH} and \cite{E2}.
Both these papers are written in the language of crossed modules. The construction
for the box-tensor product given in Section \ref{S:5} combines ideas of \cite{E1}
and \cite{NR}.

In Section \ref{S:6}, we consider the tensor product introduced by N. Inassaridze in
\cite{NI}. It is an open question whether the tensor product $G\otimes H$ introduced
by N. Inassaridze is finite whenever each of the factors is finite. In Theorems
\ref{T1} and \ref{T2}, the author of \cite{NI1} proves the finiteness of $G\otimes
H$ when the factors are finite along with some additional hypothesis. We generalize
his results on finiteness (cf. Theorem \ref{niko1}).
 We also prove the finiteness of $G\otimes H$ if both $G$ and $H$ are finite groups
and if the mutual actions are half compatible. As an application, we prove the
finiteness of the low dimensional non-abelian homology groups.

\section{Compatibility and Full Compatibility} \label{S:2}

In this section we compare compatibility and full compatibility conditions. It
should be noted that the condition \eqref{fc} for full compatibility yields 8
conditions, 4 pairs of conditions which we now give for the readers convenience.
\begin{align}
 ^{(^gh)}g' & \ = \ \;^g\big(^h(^{g^{-1}}g')\big)  & \mbox{or} \qquad  \qquad
^{(^hg)}h' & \ = \ \;^h\big(^g(^{h^{-1}}h')\big)  \label{fc1} \\
^{(^gg')}g'' & \ = \ \;^g\big(^{g'}(^{g^{-1}}g'')\big) &  \mbox{or}  \; \; \quad
\qquad  ^{(^hh')}h'' &  \ = \ \;^h\big(^{h'}(^{h^{-1}}h'')\big)  \label{fc2} \\
^{(^gg')}h & \ = \ \;^g\big(^{g'}(^{g^{-1}}h)\big) &  \mbox{or} \qquad \qquad
^{(^hh')}g & \ = \ \;^h\big(^{h'}(^{h^{-1}}g)\big)          \label{fc3} \\
^{(^gh)}h' & \ = \ \;^g\big(^{h}(^{g^{-1}}h')\big) & \mbox{or}  \qquad \qquad
^{(^hg)}g' & \ = \ \;^h\big(^g(^{h^{-1}}g')\big)       \label{fc4}
\end{align}
for all $g,g',g''\in G$ and $h,h',h''\in H$.

Next we want to show that when the groups act on themselves by conjugation, a
compatible action is fully compatible.

\begin{lemma} Let $G$ and $H$ be groups that act on themselves by conjugation and
each of which acts on the other. If the mutual actions are compatible, then they are
fully compatible.
\end{lemma}

\begin{proof}
Since the groups act on themselves by conjugation we have $^gg'=gg'g^{-1}$. Hence
$^{(^gg')}h=\,^g\big(^{g'}(^{g^{-1}}h)\big)$.
Now notice that if we replace $g,g'$ with $h,h'$ and $h$ by $g$ then we will obtain
$^{(^hh')}g=\,^h\big(^{h'}(^{h^{-1}}g)\big)$.
Hence \eqref{fc3} holds. Similarly replacing $h$ by $g''$, we obtain
$^{(^gg')}g''=\,^g(^{g'}(^{g^{-1}}g''))$.
Hence \eqref{fc2} holds.
Recall that $^hg\in G$. Since the groups act on themselves by conjugation,
$^{(^hg)}g'=\,(^hg)g'\,(^hg^{-1})=\,^h(g\;^{h^{-1}}g'g^{-1})=\,^h\big(^g(^{h^{-1}}g')\big)$.
 By replacing $g,g'$ with $h,h'$ and $h$ by $g$ in the last equation leads to
$^{(^gh)}h'=\,^g\big(^h(^{g^{-1}}h')\big)$. Thus \eqref{fc4} holds.
 By hypothesis, the mutual actions are compatible,  hence  \eqref{fc1} holds and so
all the eight conditions required for full compatibility are satisfied. Thus a
compatible action is fully compatible.
\end{proof}

 The goal of the next example is to have an action which satisfies all the 8
conditions of full compatibility and one where not all 8 conditions are satisfied
but the product formed is still finite.

\begin{example} Let $G\cong H\cong C_2^2=V_4=\{a,b,ab,e\}$, the Klein-four group.
There are 3 possible non-trivial actions of $G$ on itself.
 We will list these actions and show that one of the three actions satisfies the
full compatibility condition.
 First recall that $\Aut(V_4)\cong S_3$ where $S_3$ is the symmetric group on 3
symbols. Let $a$ and $b$ be the generators of $V_4$.
 Let $\psi_{ab} \colon V_4\rightarrow \Aut(V_4)$ be defined by $\psi_{ab}(a)=(12)$
and $\psi_{ab}(b)=(12)$.
 Similarly define $\psi_{b}$ and $\psi_{a}$ by $\psi_{b}(a)=\psi_{b}(b)=(13)$ and
$\psi_{a}(a)=\psi_{a}(b)=(23)$.
In particular $\psi_{ab}$ gives rise to the automorphism of $V_4$ that fixes $ab$.
Similarly $\psi_{b}$ and $\psi_{a}$ give
rise to automorphisms of $V_4$ that fix $b$ and $a$ respectively. In order to form
the box-tensor product we have to consider four actions:
$G$ acting on itself and on $H$ and $H$ acting on itself and on $G$. Let all these
four actions be defined by $\psi_{b}$.
We will show that this does not give a fully compatible action. For the actions
defined by $\psi_{b}$,
notice that $^{(^aa)}a=\,^{ab}a=a\neq\,^a\big(^a(^{a^{-1}}a)\big)=ab$. Similarly,
all the four actions defined
by $\psi_{a}$ are not fully compatible and this can be seen as follows:
$^{(^bb)}b=\,^{ab}b=b\neq\,^b\big(^b(^{b^{-1}}b)\big)=ab$.

It is easy to check that the four actions defined by $\psi_{ab}$ are fully
compatible. A Gap computation shows that the resulting box-tensor product is
$C_4\times C_2$. We have already seen that the four actions defined by $\psi_{b}$
are not fully compatible. Gap computation shows that the resulting product is
$C_2\times C_2$, a finite group. We have also seen that the four mutual actions
defined by $\psi_{a}$ are not fully compatible. Gap computation shows that the
resulting product is again the Klein four group.

\end{example}

\section{Basic Results}\label{S:3}

In this section we prove some basic results for the box-tensor product. These
results are well known for nonabelian tensor products and follow here as easy
corollaries to results for box-tensor products. Similar to the center of a group $G$
which fixes all elements of the group under conjugation, we can look at something
which we
will call the $G$-center of a group, if $G$ acts on a group $H$ and on itself not
necessarily by conjugation. We make the following definition.

\begin{definition} Let $G$ and $H$ be groups with $G$ acting on $G$ and $H$. \emph{The
$G$-center of $G$ with respect to $H$ }is then defined as $F_H(G)=\{g\in G \mid \;
^ga=a
\;\text{for all}\; a\in G\cup H \}$.
\end{definition}

Let $G$ and $H$ be groups with $H$ acting on $G$. In \cite{MV}, a subgroup,
called the \emph{derivative} of $G$ by $H$, was introduced. It is defined as
$D_H(G)=\left\langle g \ ^hg^{-1}\mid g\in G,h\in H\right\rangle$ and it was shown that
$D_H(G)$ is normal in $G$, provided the mutual actions are compatible. The derivative
of $G$ by $H$ is generated by the
deviations of $^hg$ from $g$. For the box-tensor product we need to introduce
another subgroup generated by the deviations of $^gg'$
from the conjugate of $g'$ by $g$. To that end, we first make a notational
convention. If $G$ is a group and $X$ is a subset of $G$, we let
$X^G=\left\langle ^gx \mid x\in X, g\in G\right\rangle$ be the normal closure of $X$ in
$G$. If $X=\{x\}$, we write $x^G$ for $\{x\}^G$.

\begin{definition}Let $G$ be a group and $\rho \colon G \rightarrow \Aut(G)$. Then the
\emph{deviational subgroup} of $G$ with respect to $\rho$ is defined as
$D_{\rho}(G)=E^G$,
where $E=\{^gg'gg'^{-1}g^{-1}\mid g,g'\in G\}$.
\end{definition}

In case of the nonabelian tensor product where $\rho \colon G\rightarrow \Inn(G)$, we
observe that $D_{\rho}(G)$ is trivial. The next proposition states various
properties of these subgroups.

\begin{prop} Let $G$ and $H$ be groups acting on themselves and on each other. Then
$F_H(G)\lhd G$. If the mutual actions are fully compatible, then
$D_{\rho}(G)$ is a subgroup of $F_H(G)$. In particular, $D_{\rho}(G)$ acts trivially
on $G$ and $H$. Furthermore, there exists $v,w\in D_{\rho}(G)$ such that
\begin{equation}\label{E:2.3.1}
^gg'=gg'g^{-1}w=vgg'g^{-1}
\end{equation}
for $g,g'\in G$.
\end{prop}

\begin{proof} Notice that $F_H(G)=\ker \sigma_G \cap \ker \rho_G$. Thus $F_H(G)\lhd
G$. To show that $D_{\rho}(G)$ is contained in $F_H(G)$, consider
$^gg'gg'^{-1}g^{-1}\in D_{\rho}(G)$. Using compatibility condition, we obtain
$^{^gg'gg'^{-1}g^{-1}}a=\;^{(^gg')}(^{gg'^{-1}g^{-1}}a)=\;^{gg'g^{-1}}(^{gg'^{-1}g^{-1}}a)=a$.
Thus $D_{\rho}(G)$ is a
subgroup of $F_H(G)$. The last part of our claim follows from observing that
$^gg'\equiv gg'g^{-1} \pmod{D_{\rho}(G)}$.
\end{proof}

It is easy to check that we have actions of $G$ and $H$ on $G\boxtimes H$ given by
$^x(g\boxtimes h)=\,^xg\boxtimes \,^xh$ for all $x\in G$ or $H$. The next proposition
generalizes \cite[Proposition 3]{BJR} to the box-tensor product.

\begin{prop}\label{P:2.4}
Let $G$ and $H$ be groups acting on each other and on themselves. If the actions are
fully compatible, then the following relations hold for all $g, g'\in G$, $h, h'\in H$.
\begin{align}
 ^g(g^{-1}\boxtimes h)=(g\boxtimes h)^{-1}=\;^h(g\boxtimes h^{-1}) &;
\label{E:2.4.1} \\
 (g\boxtimes h)(g'\boxtimes h')(g\boxtimes
h)^{-1}=\;^{[g,h]}g'\boxtimes\;^{[g,h]}h'& ; \label{E:2.4.2}\\
 (g\;^hg^{-1})\boxtimes h'=\;^{gh}(g^{-1}\boxtimes v)(g\boxtimes h)\;^{h'}(g\boxtimes
h)^{-1} & \quad
\mbox{for some }v\in D_{\rho}(H); \label{E:2.4.3} \\
g'\boxtimes (^ghh^{-1})=\;^{g'g}(w\boxtimes h)\;^{g'}(g\boxtimes h)(g\boxtimes
h)^{-1}  & \quad
\mbox{for some }w\in D_{\rho}(G); \label{E:2.4.4} \\
(g\;^hg^{-1})\boxtimes (^{g'}h'h'^{-1})=\;^{gh}(g^{-1}\boxtimes v)[g\boxtimes
h,g'\boxtimes h'] &  \quad
\mbox{for some }v\in D_{\rho}(H); \label{E:2.4.5} \\
(g\;^hg^{-1})\boxtimes (^{g'}h'h'^{-1})=\;^{g\;^hg^{-1}g'}(w\boxtimes h)[g\boxtimes
h,g'\boxtimes h']& \quad
\mbox{for some }w\in D_{\rho}(G). \label{E:2.4.6}
\end{align}
\end{prop}

\begin{proof}
To prove \eqref{E:2.4.1}, expansion using \eqref{E:1.2.1} yields
$1_{\boxtimes}=gg^{-1}\boxtimes h=\;^g(g^{-1}\boxtimes h)(g\boxtimes h)$. Similarly
using \eqref{E:1.2.2}, we obtain $1_{\boxtimes}=g\boxtimes hh^{-1}=(g\boxtimes
h)\;^h(g\boxtimes h^{-1})$, the desired result.

\

To obtain \eqref{E:2.4.2}, we expand $gg'\boxtimes hh$ in two different ways. First
we expand $gg'\boxtimes hh$ using \eqref{E:1.2.1} and then  \eqref{E:1.2.2} on the
resulting product. This together with the compatibility condition yields
$gg'\boxtimes hh'=(^gg'\boxtimes\;^gh)(^{gh}g'\boxtimes\;^{gh}h')(g\boxtimes
h)(^hg\boxtimes\;^hh')$. Similarly, applying \eqref{E:1.2.2} first and then
\eqref{E:1.2.1} yields $gg'\boxtimes hh'=(^gg'\boxtimes\; ^gh)(g\boxtimes
h)(^{hg}g'\boxtimes\;^{hg}h')(^hg\boxtimes\;^hh')$. Equating the right sides of both
equations and cancelling leads to
\[(^{gh}g'\boxtimes\;^{gh}h')(g\boxtimes h)= (g\boxtimes
h)(^{hg}g'\boxtimes\;^{hg}h')\,.\]
Left multiplication of the above by $(g\boxtimes h)^{-1}$ leads to
\[(g\boxtimes h)(^{hg}g'\boxtimes\;^{hg}h')(g\boxtimes
h)^{-1}=\;^{gh}g'\boxtimes\;^{gh}h'\,.\]
Finally, setting $^{hg}g'$ for $g'$ and $^{hg}h'$ for $h'$ gives the desired
result.

\

Next we prove \eqref{E:2.4.3}. Expanding by \eqref{E:1.2.1} and using
\eqref{E:2.3.1} yield
\begin{equation}\label{E:2.4.7}
g\;^hg^{-1}\boxtimes h'=\;^{gh}(g^{-1}\boxtimes vh^{-1}h'h)(g\boxtimes h')
\end{equation}
for some $v\in D_{\rho}(H)$. Setting $x=g^{-1}\boxtimes vh^{-1}h'h$ and expanding
$x$ by using \eqref{E:1.2.2} twice together with the fact that the elements of
$D_{\rho}(H)$ act trivially on both $H$ and $G$, we obtain $x=(g^{-1}\boxtimes
v)(g^{-1}\boxtimes h^{-1}h'h)$. Substituting this expression for $x$ into
\eqref{E:2.4.7} yields
\begin{equation}\label{E:2.4.8}
g\;^hg^{-1}\boxtimes h'=\;^{gh}(g^{-1}\boxtimes v)\;^{gh}(g^{-1}\boxtimes
h^{-1}h'h)(g\boxtimes h').
\end{equation}

Expanding $g^{-1}\boxtimes h^{-1}h'h$ using \eqref{E:1.2.2} and substituting this
expression into \eqref{E:2.4.8} leads to
\begin{equation}\label{E:2.4.9}
g\;^hg^{-1}\boxtimes h'=\,^{gh}(g^{-1}\boxtimes v)\;^{gh}(g^{-1}\boxtimes
h^{-1})\;^g(g^{-1}\boxtimes h'h)(g\boxtimes h').
\end{equation}

By expanding $g^{-1}\boxtimes h'h$ in \eqref{E:2.4.9} using \eqref{E:1.2.2} and
then applying \eqref{E:2.4.1}, we obtain
\begin{equation}\label{E:2.4.10}
g\;^hg^{-1}\boxtimes h'=\;^{gh}(g^{-1}\boxtimes v)\;^{gh}(g^{-1}\boxtimes
h^{-1})(g\boxtimes h')^{-1}\;^{gh'}(g^{-1}\boxtimes h)(g\boxtimes h').
\end{equation}

Note that by \eqref{E:2.4.2} and \eqref{E:2.4.1} we obtain $(g\boxtimes
h')^{-1}\;^{gh'}(g^{-1}\boxtimes h)(g\boxtimes h')=\;^{h'}(g\boxtimes h)^{-1}$.
Similarly, \eqref{E:2.4.1} yields $^{gh}(g^{-1}\boxtimes h^{-1})=\;^g(g^{-1}\boxtimes
h)^{-1}=g\boxtimes h$. Substituting these two expressions into \eqref{E:2.4.10}
yields the desired result.

To prove \eqref{E:2.4.4}, we proceed in a similar manner as in the previous case.
The goal is to expand the left side of \eqref{E:2.4.4} such that the same factors
appear as in the corresponding nonabelian tensor version of \eqref{E:2.4.4} plus a
correction factor involving an entry from $D_{\rho}(G)$. Expanding  with the help of
\eqref{E:1.2.2} and \eqref{E:2.3.1}, we obtain
\[g'\boxtimes\, ^ghh^{-1}=\;^g(^{g^{-1}}g'\boxtimes h)\;^{(^gh)}(g'\boxtimes
h^{-1})=\;^g(g^{-1}g'gw\boxtimes h)\;^{(^gh)}(g'\boxtimes h^{-1})\,,\]
for some $w\in D_{\rho}(G)$. Now expanding the first factor of the right
side of the
above equation with the help of \eqref{E:1.2.1} and using the compatibility
condition together with \eqref{E:2.4.1} for the second term, we arrive at
\begin{equation}\label{E:2.4.11}
g'\boxtimes\, ^ghh^{-1}=(g'gw\boxtimes h)\;^{g}(g^{-1}\boxtimes
h)\;^{[g,h]}(g'\boxtimes h)^{-1},
\end{equation}
for some $w\in D_{\rho}(G)$. Expanding the first factor on the right hand side of
\eqref{E:2.4.11} by using \eqref{E:1.2.1} gives $(g'gw\boxtimes
h)=\;^{g'g}(w\boxtimes h)\;^{g'}(g\boxtimes
h)(g'\boxtimes h)$. Using \eqref{E:2.4.1} on the second factor of \eqref{E:2.4.11}
and using \eqref{E:2.4.2} on the third factor together with cancellation leads to the
desired result.

\

Finally, we turn to the proof of \eqref{E:2.4.5} and \eqref{E:2.4.6}. First using
\eqref{E:2.4.3} and then observing that $^{(^{g'}h'{h'}^{-1})}(g\boxtimes
h)^{-1}=\;^{[g',h']}(g\boxtimes h)^{-1}$ by the compatibility condition yields

\[g\;^hg^{-1}\boxtimes\; ^{g'}h'{h'}^{-1}=\;^{gh}(g^{-1}\boxtimes v)(g\boxtimes
h)\;^{[g',h']}(g\boxtimes h)^{-1}\,.\]

By applying \eqref{E:2.4.2} to the last factor of the right hand side of the above
equation we obtain
\[g \ ^hg^{-1}\boxtimes\; ^{g'}h'{h'}^{-1}=\;^{gh}(g^{-1}\boxtimes v)(g\boxtimes
h)(g'\boxtimes h')(g\boxtimes h)^{-1}(g'\boxtimes h')^{-1}\,,\]
which is \eqref{E:2.4.5}. Similarly using \eqref{E:2.4.4} yields \eqref{E:2.4.6}.
\end{proof}

Observing that $D_{\rho}(G)$ and $D_{\rho}(H)$ are trivial for the nonabelian tensor
product, we obtain \cite[Proposition 3]{BJR} as a corollary.

\begin{corollary}\label{C:2.5} Let $G$ and $H$ be groups acting on each other and
acting on themselves by conjugation. If the actions are compatible, then the
following relations hold for all $g, g'\in G$, $h, h'\in H$.
\begin{align}
 ^g(g^{-1}\otimes h)\ = \ &(g\otimes h)^{-1}\ = \ ^h(g\otimes h^{-1});
\label{E:2.5.1}\\
 (g\otimes h)(g'\otimes h')(g\otimes h)^{-1}\ = \ & \;^{[g,h]}g'\otimes\;^{[g,h]}h';
\label{E:2.5.2}\\
 (g\;^hg^{-1})\otimes h'\ = \ & (g\otimes h)\;^{h'}(g\otimes h)^{-1};
\label{E:2.5.3} \\
 g'\otimes (^ghh^{-1})\ = \ & \;^{g'}(g\otimes h)(g\otimes h)^{-1};
\label{E:2.5.4} \\
 (g \  ^hg^{-1})\otimes (^{g'}h'h'^{-1})\ = \ &[g\otimes h,g'\otimes h'].
\label{E:2.5.5}
\end{align}
\end{corollary}

The following well-known concept of a crossed module can be found in \cite{CW}. There
it appears in relation with the third cohomology group.

\begin{definition} Let $A$ and $B$ be groups. A \emph{crossed module} is a group
homomorphism $\phi \colon A\rightarrow B$ together with an action of $B$ on $A$
satisfying
\[\phi(^ba)=b\phi(a)b^{-1} \qquad \text{and} \qquad  ^{\phi(a)}a'=aa'a^{-1}\,,\] for
all $b\in B$
and $a,a'\in A$.
\end{definition}

In \cite{BL2}, it is shown that the mapping $\phi \colon G\otimes H\rightarrow G$
defined
by $\phi(g\otimes h)=g \; ^hg^{-1}$ is a crossed module. We cannot do the same for the
box-tensor product. Recall that $D_{\rho}(G)$ acts trivially on $G$ and $H$, and
hence there exists an induced action of $G/D_{\rho}(G)$ on $G$ and $H$. Similarly for
$D_{\rho}(H)$. Setting $G/D_{\rho}(G)=\overline{G}$ and
$H/D_{\rho}(H)=\overline{H}$, we obtain the following result for box-tensor product.

\begin{prop}\label{P:2.8} Let $\phi \colon G\boxtimes H \rightarrow \overline{G}$ be
defined by $\phi(g\boxtimes h)=g \; ^h g^{-1}D_{\rho}(G)$. Then the following hold:
\begin{itemize}
  \item[(i)] $\phi$ is a homomorphism;
 \item[(ii)] there is an action of $\overline{G}$ on $G\boxtimes H$ defined by
$^x(g\boxtimes h)=\;^x g\boxtimes\;^x h$, where $x\in \overline{G}$;
 \item[(iii)] $\phi \colon G\boxtimes H \rightarrow \overline{G}$ is a crossed module.
\end{itemize}
\end{prop}

\begin{proof} To prove (i), we have to show that
\[\phi(gg'\boxtimes
h)=\phi(^{g}g'\boxtimes\; ^{g}h)\phi(g\boxtimes h)\quad \text{and} \quad
\phi(g\boxtimes hh')=\phi(g\boxtimes h)\phi(^{h}g\boxtimes\; ^{h}h')\,.\]
By the definition of $\phi$ it follows that $\phi(gg'\boxtimes h)=gg'\; ^h g'^{-1}\;
^h g^{-1}D_{\rho}(G)$. Using the compatibility condition yields

\[\phi(^{g}g'\boxtimes \;^{g}h)\phi(g\boxtimes h)= \;^{g}g'\;^{gh}g'^{-1}\;g
\;^hg^{-1}D_{\rho}(G)=\;^{g}(g'\;^{h}g'^{-1})\;g \;^hg^{-1}D_{\rho}(G)\,.\]
Now using \eqref{E:2.3.1} gives $\phi(^{g}g'\boxtimes \;^{g}h)\phi(g\boxtimes
h)=g(g'\;^{h}g'^{-1}) \;^hg^{-1}D_{\rho}(G)$, the desired result.

The second equality $\phi(g\boxtimes hh')=g\; ^{hh'}g^{-1}D_{\rho}(G)$ follows in a
similar manner.

Turning to (ii), we have to show that \[^x(gg')\boxtimes\,^xh=(^{xg}g'\boxtimes
\,^{xg}h)(^xg\boxtimes\,^xh)\quad \text{and} \quad
^xg\boxtimes\,^x(hh')=(^xg\boxtimes\,^xh)(^{xh}g\boxtimes\, ^{xh}h')\,,\]where $x\in
\overline{G}$. By \eqref{E:1.2.1} and the compatibility condition we obtain
\[^x(gg')\boxtimes\,^xh=(^{(^xgx)}g'\boxtimes\,^{(^xgx)}h)(^xg\boxtimes
\,^xh)=(^{xg}g'\boxtimes\,^{xg}h)(^xg\boxtimes\,^xh)\,.\]
 The second equality follows
in a similar manner.

\

To prove (iii), it is sufficient to check that
\[\phi(^xg\boxtimes\,^xh)=x\phi(g\boxtimes h)x^{-1}\quad \text{and}\quad  (g'\boxtimes
h')(g\boxtimes h)(g'\boxtimes
h')^{-1}=\;^{\phi(g'\boxtimes h')}(g\boxtimes h)\,,\]
where $x\in \overline{G}$. By the
compatibility condition and using \eqref{E:2.3.1}, we obtain
\[\phi(^xg\boxtimes\,^xh)=\;^xg\;(^{xh}g^{-1})D_{\rho}(G)=\;^x(g\;^hg^{-1})D_{\rho}(G)=x\phi(g\boxtimes
h)x^{-1}D_{\rho}(G)\,.\]
 Using \eqref{E:2.4.2} and again the compatibility condition yields
\[(g'\boxtimes h')(g\boxtimes h)(g'\boxtimes
h')^{-1}=\;^{(g' \; ^{h'}g'^{-1})}g\boxtimes\;^{(g' \;
^{h'}g'^{-1})}h=\;^{\phi(g'\boxtimes
h')}(g\boxtimes h)\,.\]
\end{proof}

Since $D_{\rho}(G)$ and $D_{\rho}(H)$ are trivial for the
nonabelian tensor product, we obtain the following corollary which is
\cite[Proposition
2.3]{BL2}.

\begin{corollary}Let $G$ and $H$ be groups acting on each other and acting on
themselves by conjugation. If the actions are compatible, then the following hold.
\begin{itemize}
  \item[(i)] The free product $G\ast H$ acts on $G\otimes H$ so that $^x(g\otimes
h)=(^xg\otimes\, ^xh)$ for $g\in G$, $h\in H$ and $x\in G\ast H$.
  \item[(ii)] There are homomorphisms $\lambda \colon G\otimes H\rightarrow G$ and
$\lambda' \colon
G\otimes H\rightarrow H$ such that $\lambda(g\otimes h)=g\,^hg^{-1}$ and
$\lambda'(g\otimes h)=\;^ghh^{-1}$.
  \item[(iii)] The homomorphisms $\lambda$ and $\lambda'$ with the given actions are
crossed
modules.
\end{itemize}
\end{corollary}

As an easy consequence of the defining relations for the box-tensor product, we
obtain the following expansion formulas.

\begin{lemma}\label{L:2.6}
Let $G$ and $H$ be groups and $k\in \mathbb{N}$. Then
\begin{align}
g^{k}\boxtimes h \ = \ & \prod_{i=1}^k \big( \, ^{g^{k-i}}(g\boxtimes h)\big),
\label{E:2.6.1} \\
g\boxtimes h^k \ = \ & \prod_{i=1}^k \big(\, ^{h^{i-1}}(g\boxtimes h)\big)
\label{E:2.6.2}
\end{align}
for all $g\in G$, $h\in H$.
\end{lemma}

\begin{proof}
The proof of \eqref{E:2.6.1} is by induction on $k$. The claim is obviously true for
$k=1$. Let $k>1$, then
$g^{k}\boxtimes h=gg^{k-1}\boxtimes h=\;^g(g^{k-1}\boxtimes h)(g\boxtimes h)$ by
\eqref{E:1.2.1}. Assuming that the claim is true for $k-1$, we obtain
$g^{k}\boxtimes h=\;^g\big(\prod_{i=1}^{k-1}\ ^{g^{k-1-i}}(g\boxtimes
h)\big)(g\boxtimes h)$.
This yields \eqref{E:2.6.1} immediately.
Similarly, using \eqref{E:1.2.2}, we obtain \eqref{E:2.6.2}.
\end{proof}

We need the following lemma for results in Section \ref{S:4}.

\begin{lemma}\label{L:2.10}
Let $A$ and $B$ be groups and $\phi \colon A \rightarrow B$ be a crossed module. Set
$C=A\rtimes B$ and $T=\big(1,\phi(A)\big)^C$, the normal closure of
$\big\{\big(1,\phi(a)\big)\mid a\in A\big\}$
in $C$. Then for $k\in \mathbb{N}$ there exists $t_k\in T$ such that
\[\big((1,b)(a,1)\big)^k=t_k(1,b)^k(a,1)^k\] for all $a\in A$, $b\in \phi(A)$.
\end{lemma}

\begin{proof}
Recall that the elements in a semi-direct product $A\rtimes B$ of two groups $A$ and
$B$ can be written as pairs $(a,b)$ with $a\in A$ and $b\in B$. The multiplication
in $A\rtimes B$ is then given by
\begin{equation}\label{E:2.10.1}
(a_1,b_1)(a_2,b_2)=(a_1 \;^{b_1}a_2,b_1b_2),
\end{equation}
where $a_1,a_2\in A$ and $b_1,b_2\in B$. The proof of the claim is by induction on
$k$. It is trivially true for $k=1$. Assume the claim holds for some $k-1$ where
$k\geq 2$, that is $\big((1,b)(a,1)\big)^{k-1}=t_{k-1}(1,b)^{k-1}(a,1)^{k-1}$ for
some $t_{k-1}\in T$. Thus we obtain from our hypothesis that
\begin{equation}\label{E:2.10.2}
\big((1,b)(a,1)\big)^{k}=t_{k-1}(1,b)^{k-1}(a,1)^{k-1}(1,b)(a,1).
\end{equation}
Since $T\triangleleft C$, it follows that for $x\in C$ and $u\in T$, there exists
$t\in T$ with $xu=tux$. Letting $x=(1,b)^{k-1}(a,1)^{k-1}$ and $u=(1,b)$ on the
right hand side of \eqref{E:2.10.2}, we arrive at
\[\big((1,b)(a,1)\big)^{k}=t_{k-1}t(1,b)^k(a,1)^k\,.\] Setting $t_k=t_{k-1}t$ in
the above proves our claim.
\end{proof}

Next we modify the derivative subgroup $D_H(G)$ of $G$ by $H$ for the box-tensor product. Instead of $D_H(G)$, we will consider the image of $D_H(G)$ under the homomorphism defined in Proposition \ref{P:2.8}.

\begin{definition} \label{D:deri_devi}
Let $G$ and $H$ be groups with $H$ acting on $G$ and $\rho \colon G \rightarrow
\Aut(G)$. Set $\overline{D_H(G)}=D_H(G)D_\rho(G)/D_\rho(G)$.
\end{definition}

\begin{theorem}\label{T:2.11}
Let $G$ and $H$ be groups that acting on each other fully compatibly.
If $\overline{D_H(G)}$ is locally cyclic, then $G\boxtimes H$ is  abelian.
\end{theorem}

\begin{proof}
Suppose $\overline{D_H(G)}$ is locally cyclic. By Proposition \ref{P:2.8}, we have that $\phi(G\boxtimes
H)=\overline{D_H(G)}$. Since $\phi$ is an epimorphism, we have that $G\boxtimes H
/(\ker \phi_G)\cong \overline{D_H(G)}$. Let $t_1,t_2 \in G\boxtimes H$. Since
$\overline{D_H(G)}$ is locally cyclic, we have that $\langle t_1 \ker \lambda_G ,
t_2 \ker \lambda_G \rangle$ is a cyclic subgroup of $G\boxtimes H /(\ker \phi_G)$.
Thus there exists $x \in G \boxtimes H$ so that $t_1 = x^r s_1$ and $t_2 = x^s s_2$
for some $r,s \in\mathbb{Z}$ and $s_1, s_2 \in \ker \phi_G$.
 Hence $[t_1,t_2] = [x^rs_1 , x^s s_2]$. Expanding the commutator in the first factor
 yields $[t_1 ,t_2]={}^{x^r} [s_1,x^s s_2 ]\; [x^r, x^s s_2]$. Recall that $\phi$ is
a crossed module and hence $\ker \phi$ is a central subgroup of $G\boxtimes H$. So
$[t_1, t_2] =1_\boxtimes$, for all $t_1,t_2\in G\boxtimes H$. We then conclude
$[(G\boxtimes H),(G\boxtimes H)]$ is trivial, so $G\boxtimes H$ is abelian.
\end{proof}

\section{Some Finiteness Conditions} \label{S:4}

In this section we will give a purely group theoretic proof that the box-tensor
product of two finite groups is finite.
Without the compatibility conditions, the finiteness of the tensor product $G\otimes
H$ as given in \cite{NI} is not known even
when the two groups $G$ and $H$ are finite. In Section \ref{S:2} we have seen examples where
the mutual actions are not fully compatible but the product formed is still finite,
when the two factors are finite. In this section, we give sufficient conditions
under which the box tensor product of two finite groups is finite. The main tool of
our proof is Dietzmann's
Lemma \cite{D} (see also \cite{DR} for a more accessible reference). Noting that a
subset of a group is \emph{normal}
if it contains all conjugates of its elements, Dietzmann's Lemma can be stated as
follows.

\begin{lemma}[Dietzmann's Lemma]\label{L:3.1} In any group $G$ a finite normal
subset consisting of
elements of finite order generates a finite normal subgroup of $G$.
\end{lemma}

Employing Dietzmann's Lemma as our main tool, we will show here that the box-tensor
product of two finite groups is finite using only group theoretic means. We start
with various lemmas and propositions addressing finiteness conditions.

\begin{lemma}\label{L:3.2}Let $A$ and $B$ be groups and $\phi \colon A \rightarrow
B$ be a
crossed module. Set $C=A\rtimes B$.
Let $b\in B$ be such that $b^k=1$ for some $k\in \mathbb{N}$ and assume that
$\prod_{i=1}^k\ ^{b^{k-i}}a=1$ for some $a\in A$. Then in $C$ we have
$(1,1)=[(1,b^{-1})(a,1)]^k$.
\end{lemma}

\begin{proof}
Using the hypothesis and the multiplication in $C$ as given in \eqref{E:2.10.1}, we
obtain
\[
1_C=\big(\prod_{i=1}^k\ ^{b^{k-i}}a,1\big)=\prod_{i=1}^k (^{b^{k-i}}a,1)=\prod_{i=1}^k
(1,b^{k-i})(a,1)(1,b^{k-i})^{-1}.
\]
Notice that $(1,b^{k-1})=(1,b^{-1})$ and $(1,b^{k-i})^{-1}(1,b^{k-i-1})=(1,b^{-1})$.
This yields \[\prod_{i=1}^k (1,b^{k-i})(a,1)(1,b^{k-i})^{-1}=[(1,b^{-1})(a,1)]^k\,,\]
the desired result.
\end{proof}

\begin{prop}\label{P:3.3}
Let $A$ and $B$ be groups and $\phi \colon A\rightarrow B$ be a crossed module.
Consider
$C=A\rtimes B$, where $B$ acts on $A$ as given by the crossed module. If $B$ is
finite, then $T=\big(1,\phi(A)\big)^C$ is finite.
\end{prop}
\begin{proof} Since $\big(1,\phi(A)\big)=\big\{\big(1,\phi(x)\big),x\in A\big\}$ is
a subset of $(1,B)$, it
is finite, say $|\big(1,\phi(A)\big)|$. Then there exist $x_1,x_2,\ldots ,x_n\in A$
such
that
\[\big(1,\phi(A)\big)=\big\{\big(1,\phi(x_1)\big),\big(1,\phi(x_2)\big),\ldots
,\big(1,\phi(x_n)\big)\big\}\,.\]
Let $N_i= \big(1,\phi(x_i)\big)^C$. We observe that $N_i\triangleleft C$. Setting
$T=\prod_{i=1}^n N_i$, it follows that $T\triangleleft C$ as the product of finitely
many normal subgroups. To prove that $T$ is finite, it suffices to show that $N_i$
is finite for $1\leq i\leq n$. Consider
$S_i=\big\{(y,b)\big(1,\phi(x_i)\big)(y,b)^{-1}\mid (y,b)\in
C\big\}$. We have $N_i=\left\langle S_i\right\rangle$. Observing that $S_i$ is a normal
set and that each element in $S_i$ has finite order, it follows from Lemma
\ref{L:3.1} that $N_i$ is finite if we can show that $S_i$ is finite.

We first show that the set
$S_i^{*}=\{(^{\phi(y)b}x_i,1)(^bx_i^{-1},1)\big(1,\phi(^bx_i)\big)\mid y\in A, b\in
B\}$ is
finite. For given $x_i$, the product of the second and third factor of an element in
$S_i^{*}$ takes at most $|B|$ values. Since $\phi(y)b\in B$, the first factor takes
at most $|B|$ values for fixed $b\in B$. Thus $|S_i^{*}|\leq |B|^2$, and hence
$S_i^{*}$ is finite. We claim now that $S_i=S_{i}^{*}$ by establishing
\begin{equation}\label{E:3.3.1}
(y,b)\big(1,\phi(x_i)\big)(y,b)^{-1}=(^{\phi(y)b}x_i,1)(^bx_i^{-1},1)\big(1,\phi(^bx_i)\big),
\end{equation}
for all $y\in A$ and $b\in B$ and hence $S_i$ is finite as required. It remains to
be shown that \eqref{E:3.3.1} holds. \par  Recalling that
$(y,b)^{-1}=(^{b^{-1}}y^{-1},b^{-1})$ and that $\phi$ is a crossed module with
$\phi(^bx_i)=b\phi(x_i)b^{-1}$, multiplication in $C$ then yields
\begin{equation}\label{E:3.3.2}
(y,b)\big(1,\phi(x_i)\big)(y,b)^{-1}=\big(y \ ^{\phi(^bx_i)}y^{-1},\phi(^bx_i)\big).
\end{equation}

Observing that $\phi$ is a crossed module with
$^{\phi(^bx_i)}y^{-1}=\;^bx_iy^{-1} \ ^bx_i^{-1}$, we obtain for the
$A$-component of the right hand side of \eqref{E:3.3.2} that
\[y \ ^{\phi(^bx_i)}y^{-1}=y \ ^bx_iy^{-1}\ ^bx_i^{-1}=\;^{\phi(y)b}x_i \
^bx_i^{-1}\,.\]
Substituting this into the right hand side of \eqref{E:3.3.2} yields
\begin{equation}\label{E:3.3.3}
\big(y \ ^{\phi(^bx_i)}y^{-1},\phi(^bx_i)\big)=\big(^{\phi(y)b}x_i \
^bx_i^{-1},\phi(^bx_i)\big).
\end{equation}

On the other hand, multiplication in $C$ leads to
\begin{equation}\label{E:3.3.4}
(^{\phi(y)b}x_i,1)(^bx_i^{-1},1)\big(1,\phi(^bx_i)\big)=\big(^{\phi(y)b}x_i \
^bx_i^{-1},\phi(^bx_i)\big)
\end{equation}
for the right hand side of \eqref{E:3.3.1}. By \eqref{E:3.3.2},
\eqref{E:3.3.3} and \eqref{E:3.3.4}, it follows that \eqref{E:3.3.1} holds, proving
our claim.
\end{proof}

In the next proposition we establish that the generating set for the box-tensor
product $G\boxtimes H$ is a finite normal set provided $G$ and $H$ are finite
groups.

\begin{prop}\label{P:3.4}
Let $G$ and $H$ be groups acting on each other fully compatibly. Consider
$Y=\{g\boxtimes
h \mid g\in G, h\in H\}$, the generating set for $G\boxtimes H$. If $G$ and $H$ are
finite, then $Y$ is a finite normal set. In particular, each $g\boxtimes h\in Y$ has
finitely many conjugates in $Y$.
\end{prop}

\begin{proof}
If $G$ and $H$ are finite, then obviously $Y=\{g\boxtimes h \mid g\in G, h\in H\}$ is
finite. We need to show that $x(g\boxtimes h)x^{-1}\in Y$ for any
$g\boxtimes h\in Y$ and $x\in G\boxtimes H$. First, consider $x=u\boxtimes v$, $u\in
G$, $v\in H$. Using \eqref{E:2.4.2} and \eqref{fact} we obtain
\begin{equation}\label{E:3.4.1}
x(g\boxtimes h)x^{-1}=\;^{u\;(^vu^{-1})}g\boxtimes\;^{(^uv)\;v^{-1}}h,
\end{equation}
and hence $x(g\boxtimes h)x^{-1}\in Y$.
By the definition of the box-tensor product and \eqref{E:2.4.1}, every $x\in
G\boxtimes H$ can be written as a finite product of elements in $Y$, that is
$x=\prod_{i=1}^k(u_i\boxtimes v_i)$, where $u_i\in G$ and $v_i\in H$. By
\eqref{E:3.4.1} and induction on the number of factors in the product for $x$, it
follows that $x(g\boxtimes h)x^{-1}\in Y$.
\end{proof}

Next we will establish that certain elements of the generating set of $G\boxtimes H$
have finite order provided $G$ is finite.

\begin{prop}\label{P:3.5}
Let $G$ and $H$ be groups acting on each other fully compatibly. If $G$ is finite, then
$s\boxtimes h$ has finite order for all $s\in D_H(G)$ and $h\in H$.
\end{prop}
\begin{proof} Let $\phi \colon G\boxtimes H \rightarrow \overline{G}$ be the crossed
module defined in Proposition \ref{P:2.8} and
consider $C=(G\boxtimes H)\rtimes \overline{G}$. By Proposition \ref{P:3.3}, it
follows that $T=\big(1,D_H(G)\big)^C$ is finite. Let
$s\in D_H(G)$. Since $G$ is finite, there exists $k\in \mathbb{N}$ such that
$s^k=1$. By \eqref{E:2.6.1} we obtain,
$1_{\boxtimes}=1\boxtimes h=s^k\boxtimes h=\prod_{i=1}^k\ ^{s^{k-i}}(s\boxtimes h)$.
This together with Lemma \ref{L:3.2} yields the following in $C$:
\begin{equation}\label{E:3.5.1}
1=[(1,s^{-1})(s\boxtimes h,1)]^k.
\end{equation}
Applying Lemma \ref{L:3.2} to \eqref{E:3.5.1} gives $1_C=t(s\boxtimes h,1)^k$, where
$t\in \big(1,D_H(G)\big)^C$. It follows by Proposition \ref{P:3.3} that $t$ has
finite order
and consequently $s\boxtimes h$ has finite order.
\end{proof}

Now we will show that certain elements of the generating set of $G\boxtimes H$ have
finite order provided $H$ is finite.

\begin{prop}{\label{P:3.6}}
Let $G$ and $H$ be groups acting on each other fully compatibly. If $H$ is finite, then
$g\boxtimes h$ has finite order for all $g\in G$ and $h\in F_G(H)$.
\end{prop}
\begin{proof} Since $H$ is finite, there exists a $k\in \mathbb{N}$ such that
$h^{k}=1$. Using \eqref{E:2.6.2} and observing that $^{h^n}a=a$, for all $a\in G\cup
H$ and any $n\in \mathbb{N}$ we obtain
\[1=g\boxtimes 1=g\boxtimes h^k=(g\boxtimes h)^k\,,\]
the desired result.
\end{proof}

\begin{corollary}\label{C:3.7}
Let $G$ and $H$ be finite groups acting on each other fully compatibly. Set
 $X=\{g\boxtimes h \mid g\in D_H(G)\;\text{or}\;\; h\in F_G(H)\}$, then $X^{G\boxtimes
H}$ is finite.
\end{corollary}

\begin{proof}
By Propositions \ref{P:3.5} and \ref{P:3.6}, it follows that the elements of $X$
have finite order.
Now using Proposition \ref{P:3.4} and Lemma \ref{L:3.1} gives the desired result.
\end{proof}

Now we are ready to prove the main theorem of this section.

\begin{theorem}\label{T:3.8} Let $G$ and $H$ be groups acting on each other
fully compatibly. If $G$ and $H$ are finite, then $G\boxtimes H$ is finite.
\end{theorem}

\begin{proof} First, we will prove that $g\boxtimes h$ has finite order for all
$g\in G$ and $h\in H$. Observe that $^hg\equiv g \pmod{D_H(G)}$ and $^hh^n\equiv
h^n \pmod{F_G(H)}$ for all $n\in \mathbb{N}$. Thus one can easily verify that
$^hg\boxtimes\; ^hh^{n}=(g\boxtimes h^{n})v_1$, where $v_1\in X^{G\boxtimes H}$.
Since $H$ is finite, there exists a $k\in \mathbb{N}$ such that $h^k=1$. By the
above observation and \eqref{E:1.2.2} we obtain,
\[
1_\boxtimes= g\boxtimes 1= g\boxtimes h^k=(g\boxtimes h)(g\boxtimes h^{k-1})v_1 \,.
\]
Using \eqref{E:1.2.2} and induction on $k$ yields $1_{\boxtimes}= (g\boxtimes
h)^{k}v$, where $v=v_{k-1}\cdots v_1\in X^{G\boxtimes H}$. By Corollary \ref{C:3.7},
$v$ has finite order. This implies $g\boxtimes h$ has finite order. Now using
Proposition \ref{P:3.4} together with Lemma \ref{L:3.1} completes the proof.
\end{proof}

Observing again that the nonabelian tensor product is a special case of the
box-tensor product, we obtain as a corollary another purely group
theoretic proof of the finiteness of the nonabelian tensor product of two finite
groups.

\begin{corollary}\label{C:3.9} Let $G$ and $H$ be groups acting on
themselves by conjugation and each of which acts on the other. If the actions are compatible and $G$ and $H$ are finite,
then $G\otimes H$ is finite.
\end{corollary}

\section{A General Construction} \label{S:5}

In this section we give a general construction for the box-tensor product. Our goal
is to describe the box-tensor product $G\boxtimes H$ as a section of the free
product $G\ast H$ of two groups $G$ and $H$, where $G$ and $H$ act on each other and
on themselves in a compatible way. Consider the following subset of $G\ast H$:
\[R=\big\{^x[g,h]^{-1}[^xg,\;^xh],\; ^y[g',h']^{-1}[^yg',\;^yh'] \mid g,g',x\in
G\backslash \{1\} \ \text{and} \  h,h',y\in H\backslash \{1\}\big\}\,.\]
Note that we write $^z[g,h]$ for $z[g,h]z^{-1}$ with $z\in G\cup H$, whenever this
is not
ambiguous. Let $\eta(G,H)=(G\ast H)/R^{G\ast H}$ and $\tau(G,H)=[G,H]/R^{G\ast H}$,
observing that $R^{G\ast H}\lhd [G,H]$. Note that the normality of $R^{G\ast H}$ in
$[G,H]$ follows from $[G,H]\lhd G\ast H$ and $R^{G\ast H}\lhd G\ast H$.

\

The following diagram is the key in reaching our goal:

\begin{equation*}
\xymatrix@+20pt{
\tau(G,H)\ \ar@{>->}[r]^{\mu} \ar@{->}[d]^-{\alpha}
&\eta(G,H)\ar@{->>}[r]^-{\overline{\sigma}}
\ar@{->}[d]^-{\beta}
 &G\times H\ar@{=}[d] \\
G\boxtimes H\ \ar@{>->}[r]^-{\mu'} &\big((G\boxtimes H)\rtimes H\big)\rtimes G
\ar@{->>}[r]^-{\nu'} &G\times H.
}
\end{equation*}

In the following, we will establish the exactness of the rows of the above diagram and
the existence of the vertical mappings. We start with a lemma.

\begin{lemma}\label{L:4.1}
There is a homomorphism $\psi \colon G\boxtimes H\rightarrow \eta(G,H)$ defined by
$\psi(g\boxtimes h)=[g,h]R^{G\ast H}$ for all $g\in G, h\in H$.
\end{lemma}

\begin{proof}
First note that the canonical homomorphism $\iota_{G} \colon G\rightarrow \eta(G,H)$ is
injective. To see this, we observe that the canonical projection
$\pi \colon G\ast H\rightarrow G$ sends each element of $R^{G\ast H}$ to the
identity and
thus induces a left inverse $\pi' \colon \eta(G,H)\rightarrow G$ of $\iota_{G}$. Thus
$\iota_{G}$ is injective. Similarly, $\iota_{H} \colon  H\rightarrow \eta(G,H)$ is
injective. So we can identify $G$ and $H$ with their images in $\eta(G,H)$.

It remains to be shown that the homomorphism $\psi$ is well defined. This means
that \eqref{E:1.2.1} and \eqref{E:1.2.2}, the defining relations of $G\boxtimes
H$ must map to the identity of $\eta(G,H)$. Note that in any group $U$ the
following familiar commutator identities always hold:
\begin{equation}\label{E:4.1.1}
[uv,w]=\;^u[v,w][u,w]
\end{equation}
and
\begin{equation}\label{E:4.1.2}
[u,vw]=[u,v]\;^v[u,w]
\end{equation}
for all $u,v,w\in U$. Using \eqref{E:4.1.1}, we obtain in $G\ast H$ that
\begin{equation}\label{E:4.1.3}
 [gg',h]([^gg',\;^gh][g,h])^{-1}=g[g',h]g^{-1}[^gg',^gh]^{-1}.
\end{equation}

 Similarly, using \eqref{E:4.1.2} yields
\begin{equation}\label{E:4.1.4}
 [g,hh']([g,h][^hg,^hh'])^{-1}=[g,h](h[g,h']h^{-1}[^hg,\;^hh']^{-1})[g,h]^{-1}
\end{equation}
for all $g,g'\in G$ and $h,h'\in H$. As a consequence of \eqref{E:4.1.3} and
\eqref{E:4.1.4} we have that $\psi$ maps
$(gg'\boxtimes h)(g\boxtimes h)^{-1}(^gg'\boxtimes\;^gh)^{-1}$ and $(g\boxtimes
hh')(^hg\boxtimes\;^hh')^{-1}(g\boxtimes h)^{-1}$ to the identity in $\eta(G,H)$.
Thus $\psi$ is a well-defined homomorphism.
\end{proof}

 Now we are ready to prove our main theorem.

\begin{theorem}\label{T:4.2}
Let $G$ and $H$ be groups acting on each other. If the mutual actions are fully
compatible, then we have $\eta(G,H)\cong \big(\tau(G,H)\rtimes H\big)\rtimes G$ and
$\tau(G,H)\cong
G\boxtimes H$.
\end{theorem}

\begin{proof}
We proceed by showing that the rows in the diagram at the beginning of this section
are exact and that the vertical homomorphisms exist. Then we will establish that two
of the vertical homomorphisms are indeed isomorphism. With the help of the Short
Five Lemma \cite{CW} we will obtain that the third vertical map is also an
isomorphism. Our claim then follows.

\

We first show that the two rows are exact. The exactness of the bottom row is
obvious. Turning to the top row, consider the canonical map $\sigma \colon G\ast
H\rightarrow G\times H$ defined by $\sigma(g_1h_1g_2h_2\cdots g_nh_n)=(g_1g_2\cdots
g_n,h_1h_2\cdots h_n)$. Clearly $\sigma$ is onto and $\ker \sigma=[G,H]$.
Furthermore, $\sigma(R)=(1,1)$, and hence we get an induced map
$\overline{\sigma} \colon \eta(G,H)\rightarrow G\times H$. This shows the exactness
of the
top row.

\

Now we will turn to showing the existence of the homomorphism $\beta$ in the above
diagram. First recall that we have actions of $G$ and $H$ on $G\boxtimes H$ given
by $^x(g\boxtimes h)=\;^xg\boxtimes\,^xh$ for all $x\in G$ or $H$. We will use
this action to form the semi-direct product $(G\boxtimes H)\rtimes H$. Observe that
$G$ acts on $(G\boxtimes H)\rtimes H$ by $^g(b,h)=\big(^gb (g\boxtimes h), h\big)$,
where
$b\in G\boxtimes H$. We claim that this is a well-defined action. Towards that
end, let $g,g'\in G$ and observe that $^{gg'}(b,h)=\big(\,^{gg'}b (gg'\boxtimes h),
h\big)$.
On the other hand $^g\big(\, ^{g'}(b,h)\big)=\big(\,^{gg'}b\;^g(g'\boxtimes
h)(g\boxtimes h), h\big)$.
As a consequence of \eqref{E:1.2.1} it follows that this is a well-defined action.
Using this action, we can form the semi-direct product $\big((G\boxtimes H)\rtimes
H\big)\rtimes G$. Let $\phi \colon G\ast H\rightarrow \big((G\boxtimes H)\rtimes
H\big)\rtimes G$ be
the homomorphism defined by $\phi(g)=(1,1,g)$ and $\phi(h)=(1,h,1)$ for all $g\in
G$, $h\in H$. Next we wish to show that $\phi([g,h])=(g\boxtimes h,1,1)$ and
$\phi(R)=(1,1,1)$. Multiplication defined in $\big((G\boxtimes H)\rtimes
H\big)\rtimes G$
as given by \eqref{E:2.10.1} yields
\begin{equation}\label{E:4.2.1}
\phi([g,h])=\phi(ghg^{-1}h^{-1})=(g\boxtimes h,h,g)(g^{-1}\boxtimes
h^{-1},h^{-1},g^{-1}).
\end{equation}

Observe that $(g\boxtimes h,h,g)(g^{-1}\boxtimes h^{-1},h^{-1},g^{-1})=\Big((g\boxtimes
h,h)  \ \big(\, ^g(g^{-1}\boxtimes h^{-1})(g\boxtimes h^{-1}),h^{-1}\big),1\Big)$.
Now using
\eqref{E:2.4.1} leads to $\Big((g\boxtimes h,h) \ \big(\,^g(g^{-1}\boxtimes
h^{-1})(g\boxtimes
h^{-1}),h^{-1}\big),1\Big)=\big((g\boxtimes h,h)(1,h^{-1}),1\big)$. Substituting
these equalities
into the right hand side of \eqref{E:4.2.1} yields $\phi([g,h])=(g\boxtimes h,1,1)$.

Proceeding as above and noting that $(1,1,x^{-1})(^xg\boxtimes\;^xh,1,1)=(g\boxtimes
h,1,1)$ we obtain \[\phi(x[g,h]^{-1}x^{-1}[^xg,^xh])=(1,1,x)\big((g\boxtimes
h)^{-1},1,1\big)(g\boxtimes h,1,x^{-1})=(1,1,1)\]
 for all $x\in G\backslash \{1\}$.
Similarly we can prove that $\phi(^x[g,h]^{-1}[^xg,^xh])=(1,1,1)$, for all $x\in
H\backslash \{1\}$. Hence we have an induced map $\beta \colon \eta(G,H)\rightarrow
\big((G\boxtimes H)\rtimes H\big)\rtimes G$. It follows that the homomorphism
$\beta$ exists.

 It remains to be shown that $\alpha$ is an
isomorphism. By Lemma \ref{L:4.1} we obtain that $\beta\psi$ is the identity on
$G\boxtimes H$. Identifying $\tau(G,H)$ and $G\boxtimes H$ with their images under
$\mu$ and $\mu'$, respectively, it follows that $\beta$ maps $\tau(G,H)$
isomorphically onto $G\boxtimes H$. Hence $\alpha$ is an isomorphism.
This completes the proof.
\end{proof}

As a consequence of Theorem \ref{T:4.2} and Theorem \ref{T:3.8}, we obtain the
following corollary.

\begin{corollary}
Let $G$ and $H$ be groups acting on each other fully compatibly. If $G$ and $H$ are
finite, then $\eta(G,H)$ is finite.
\end{corollary}

\begin{proof}
By Theorem \ref{T:3.8} we have $|G\boxtimes H|$ is finite if $G$ and $H$ are finite.
Thus by Theorem \ref{T:4.2} we obtain $|\eta(G,H)|=|G\boxtimes H||G||H|$. Hence
$\eta(G,H)$ is finite if $G$ and $H$ are finite.
\end{proof}

\section{Some finiteness conditions for Inassaridze's tensor product}\label{S:6}

In this section we introduce a tensor product which is a special case of Inassaridze's tensor product but which is more general than the nonabelian tensor product as introduced in \cite{BL1} and \cite{BL2}. We will use the methods developed for box-tensor product to give
some finiteness conditions for Inassaridze's tensor product. Let $G$ and $H$ be
groups acting on each other. Assume that $G$ and $H$ act on themselves by
conjugation. Then Inassaridze's tensor product $G\otimes H$ is a group generated by
the symbols $g\otimes h$ subject to the following three relations:
\begin{align*}
gg'\otimes h \ = \ & (^g g' \otimes \;^gh)(g \otimes h), \\
(g\otimes hh')\ = \ &(g\otimes h)(^hg\otimes \;^hh'), \\
(g\otimes h)(g'\otimes h')(g\otimes h)^{-1} \ = \ &(^{[g,h]}g'\otimes \;^{[g,h]}h'),
\end{align*}
for all $g,g'\in G$ and $h,h'\in H$, where $[g,h]=ghg^{-1}h^{-1}\in G*H$. It should
be noted at this point that in \cite{NI}, the author requires four conditions but
the above three conditions are sufficient. For Inassaridze's tensor product no
compatibility conditions are
required, but \eqref{fact} holds.

In \cite{NI1}, the author proves the following two theorems.

\begin{theorem}\cite[Theorem 9]{NI1}\label{T1}
Let $G$ and $H$ be groups acting on themselves by conjugation and each of which acts
on the other. Suppose the action of $H$ on $G$ is trivial. If $H$ is soluble, then
$G\otimes H$ is finite.
\end{theorem}

\begin{theorem}\cite[Theorem 12]{NI1}\label{T2}
Let $G$ and $H$ be groups acting on themselves by conjugation and each of which acts
on the other. Suppose the action of $H$ on $G$ is trivial. If $[G,H]^n$ is abelian
for some $n\geq 1$, then $G\otimes H$ is finite.
\end{theorem}

The next theorem which is a generalization of the above two theorems is now an easy
consequence of the methods developed in proving the finiteness of the box-tensor
product.

\begin{theorem}\label{niko1}
Let $G$ and $H$ be finite groups acting on themselves by conjugation and each of
which acts on the other. If $G$ acts on $H$ trivially, then $G\otimes H$ is finite.
\end{theorem}
\begin{proof} First we will show that the set $T=\{g\otimes h \mid g\in G, h\in H\}$
is a finite normal set. Since \eqref{E:2.4.2} and \eqref{fact} hold for
Inassaridze's product, it follows from the proof of Proposition \ref{P:3.4} that $T$
is a finite normal set.
 Now we will show that $g\otimes h$ has finite order for all $g\in G$, $h\in H$.
Let $n\in \mathbb{N}$ be such that $g^n=1$. Using \eqref{E:2.6.1} and the fact that
$G$ acts trivially on $H$ and on itself by conjugation, we obtain that
$1 =(1\otimes h)= (g^n \otimes h)=(g\otimes h)^n$. Now an application of Dietzmann's
Lemma (Lemma \ref {L:3.1}) completes the proof.
\end{proof}

In the next theorem we consider a product which is more general than the nonabelian tensor product.
It is more general because we do not require the mutual actions to be compatible.
We prove the finiteness of this product when the two factors are
finite and  when the mutual actions are half compatible. So the following theorem is a generalization
of the Ellis-Thomas's theorem on the finiteness of the nonabelian tensor product as given
in \cite{E} and \cite{VT}. It is also a generalization of Theorem \ref{niko1}
because if one of the groups is acting trivially on the other, then the mutual actions are half compatible.

\begin{theorem}\label{niko2}
Let $G$ and $H$ be finite groups acting on each other. If the mutual actions are half compatible, then $G\otimes
H$ is finite.
\end{theorem}

\begin{proof} Without loss of generality one can assume that the action of $G$ on
$H$ is compatible. Let $X$ be a normal subgroup of $G$
generated by the elements $^{(^gh)}g'\;^{ghg^{-1}}g'^{-1}$  for all $g,g'\in G$ and
$h\in H$. We claim that $X$ acts trivially on $H$. To prove our claim, it suffices
to check that $^{(^{(^gh)}g')}h'=\; ^{(^{ghg^{-1}}g')}h'$ , for all $g,g'\in G$ and
$h,h'\in H$. Using \eqref{fact} and the fact that $G$ acts compatibly on $H$, we
obtain
\begin{equation}\label{E:6.2.1}
^{(^{(^gh)}g')}h'= \;^{(^gh)g'(^gh^{-1})}h'= \;^{ghg^{-1}g'gh^{-1}g^{-1}}h'.
\end{equation}
Using the conjugation action of $G$ on itself and the compatible action of $G$ on
$H$, we arrive at
\begin{equation}\label{E:6.2.2}
^{(^{ghg^{-1}}g')}h'= \;^{^{gh}(g^{-1}g'g)}h'= \;^{g\,^{h}(g^{-1}g'g)g^{-1}}h'=
\;^{gh(g^{-1}g'g)h^{-1}g^{-1}}h'.
\end{equation}
So our claim follows from \eqref{E:6.2.1} and \eqref{E:6.2.2}. Let $X'$ be the
normal subgroup of $G$ generated by the elements $x$ and $^hx'$ for all $x, x'\in
X$, $h\in H$. Notice that $X\subseteq X'$. Since $X$ acts trivially on $H$ and the
action of $G$ on $H$ is compatible, $X'$ also acts trivially on $H$. Observe that
$X'$ is closed under the action of $H$. Hence we have an action of $G/X'$ on $H$ and
this action is compatible. We also have an action of $H$ on $G/X'$ which is also
compatible according to the definition of $X'$. Therefore, the nonabelian tensor
product $(G/X')\otimes H$ is finite. Using the following exact sequence which
appears in \cite[Theorem 1.\,(a)]{NI}
\[
X'\otimes H\to G\otimes H \to (G/X')\otimes H \to 1
\]
and Theorem \ref{niko1}, we obtain that $G\otimes H$ is finite.
\end{proof}

\begin{Remark}
The normal subgroup $X$ is denoted by Comp\;$G(H)$ in \cite{NI1}. We are using $X$
for notational convenience.
\end{Remark}

\begin{corollary}\label{}
Let $G$ and $A$ be finite groups acting on each other. If the mutual actions
are half compatible, then the nonabelian
homology groups $H_i(G, A)$ are finite when $i=0,1$.
\end{corollary}
\begin{proof} This result is straightforward, since $H_1(G, A)=\ker f$ and $H_0(G,
A)=\coker f$,
where $f \colon  G\otimes A \to A/A'$, $f(g\otimes a)=\;^gaa^{-1}A'$, and $A'$ is the
normal subgroup of $A$ generated by the elements $^{(^ag)}a' \;^{aga^{-1}}a'^{-1}$
for all $a,a'\in A$ and $g\in G$.
\end{proof}

\

\end{document}